\documentclass[a4paper,12pt]{article}

\usepackage[utf8]{inputenc}
\usepackage[english]{babel}
\usepackage{amsmath, amssymb, amsthm}
\usepackage{enumerate}
\usepackage[colorlinks=true,linkcolor=blue,citecolor=blue]{hyperref}

\usepackage{tikz}
\usetikzlibrary{fadings}
\tikzfading[name=fade out,
            inner color=transparent!0,
            outer color=transparent!100]
\usetikzlibrary{calc,arrows,decorations.pathreplacing,fadings,3d,positioning}

\title{Diophantine exponents of lattices.
       \thanks{ This research was supported by RSF grant 14-11-00433}}
\author{Oleg\,N.\,German}
\date{}

\theoremstyle{definition}
\newtheorem{definition}{Definition}

\newtheorem*{notation*}{Notation}

\theoremstyle{remark}

\newtheorem*{remark*}{Remark}

\theoremstyle{plain}
\newtheorem{theorem}{Theorem}

\newtheorem{corollary}{Corollary}
\newtheorem*{statement*}{Statement}
\newtheorem*{corollary*}{Corollary}

\renewcommand{\phi}{\varphi}

\renewcommand{\vec}[1]{\mathbf{#1}}
\renewcommand{\geq}{\geqslant}
\renewcommand{\leq}{\leqslant}

\newcommand{\e}{\varepsilon}

\newcommand{\R}{\mathbb{R}}
\newcommand{\Z}{\mathbb{Z}}
\newcommand{\Q}{\mathbb{Q}}

\newcommand{\La}{\Lambda}

\newcommand{\cL}{\mathcal{L}}
\newcommand{\cM}{\mathcal{M}}

\newcommand{\cP}{\mathcal{P}}

\begin{document}

\maketitle

\begin{abstract}
  In this paper we define Diophantine exponents of lattices and investigate some of their properties. We prove transference inequalities and construct some examples with the help of Schmidt's subspace theorem.
\end{abstract}


\section{Introduction}

There is a large variety of problems where different kinds of \emph{Diophantine exponents} naturally arise. In a rather general setting we have $n$ linearly independent linear forms $\pmb\ell_1(\vec z),\ldots,\pmb\ell_n(\vec z)$ in $d$ real variables, $n<d$. And the question is ``\emph{how small can the $n$-tuple $(\pmb\ell_1(\vec z),\ldots,\pmb\ell_n(\vec z))$ be if $\vec z$ ranges through nonzero integer points?}''. There are two classical ways to measure the ``size'' of this $n$-tuple. The first one is to consider an arbitrary norm, say, the sup-norm, and the second one is to consider the product of the absolute values of the entries. Then, we are to figure out how fast this quantity can tend to zero with the growth of the ``size'' of $\vec z$.

\paragraph{Two examples for $n<d$.}

The simplest examples illustrating these two approaches are the problem of simultaneous approximation of two real numbers and the famous Littlewood conjecture (see also \cite{cassels_GN}, \cite{gruber_lekkerkerker}). They both deal with two forms $\pmb\ell_1(\vec z),\pmb\ell_2(\vec z)$ in three variables with coefficients written in the rows of
\[ \begin{pmatrix}
     \theta_1 & 1 & 0      \\
     \theta_n & 0 & 1
   \end{pmatrix}. \]

Let $|\,\cdot\,|$ denote the sup-norm. Then the supremum of real $\gamma$ such that the inequality
\[ \max_{i=1,2}|\pmb\ell_i(\vec z)|\leq|\vec z|^{-\gamma} \]
admits infinitely many solutions in $\vec z\in\Z^3$ is called the \emph{Diophantine exponent} of the pair $(\theta_1,\theta_2)$, and it describes how well $\theta_1$ and $\theta_2$ can be simultaneously approximated with rationals which have same denominator.

On the other hand, the famous Littlewood conjecture claims that for each $\e>0$ the inequality
\[ \prod_{i=1,2}|\pmb\ell_i(\vec z)|\leq\e z_1^{-1} \]
admits infinitely many solutions in $\vec z=(z_1,z_2,z_3)\in\Z^3$, $z_1\neq0$. Similar to the case of simultaneous approximation, the \emph{multiplicative Diophantine exponent} of the pair $(\theta_1,\theta_2)$ is defined as the supremum of real $\gamma$ such that the inequality
\[ \prod_{i=1,2}|\pmb\ell_i(\vec z)|^{1/2}\leq z_1^{-\gamma} \]
admits infinitely many solutions in $\vec z=(z_1,z_2,z_3)\in\Z^3$, $z_1\neq0$. 

\paragraph{The case $n=d$.}

In the examples given above it was important that the number of linear forms is strictly less than the dimension of the ambient space. This guaranteed that the region where we look for integer points is at least unbounded. But if $n=d$, and the forms are linearly independent, the ``norm'' approach gives us a bounded region, a parallelepiped, which is good for considering something like consecutive minima, but does not allow to define any kinds of Diophantine exponents. However, the ``product'' approach appears to be rather fruitful from this point of view. It leads us to the concept of a \emph{Diophantine exponent of a lattice}.

\section{Lattice exponents}

Let us remind (see \cite{cassels_swinnerton_dyer}) that the Littlewood conjecture is closely connected to the so called Oppenheim conjecture for linear forms, which deals with the lattice
\begin{equation} \label{eq:lattice}
  \La=\Big\{ \big(\pmb\ell_1(\vec z),\ldots,\pmb\ell_d(\vec z)\big)\, \Big|\ \vec z\in\Z^d \Big\},
\end{equation}
where $\pmb\ell_1(\vec z),\ldots,\pmb\ell_d(\vec z)$ are linearly independent linear forms in $d$ variables. It claims that for $d\geq3$ the quantity
\begin{equation} \label{eq:norm_minimum}
  N(\La)=\inf_{\vec z\in\Z^d\backslash\{\vec 0\}}\prod_{1\leq i\leq d}|\pmb\ell_i(\vec z)|,
\end{equation}
which is called the \emph{norm minimum} of $\La$, is positive if and only if $\La$ is similar modulo the action of the group of diagonal matrices to the lattice of $\cM$, where $\cM$ is a complete module in a totally real algebraic extension of $\Q$ of degree $d$ (cf. \cite{borevich_shafarevich}). Thus, if we define for each $\vec x=(x_1,\ldots,x_d)\in\R^d$ the quantity
\[ \Pi(\vec x)=\prod_{1\leq i\leq d}|x_i|^{1/d}, \]
we can see that the Oppenheim conjecture proposes a criterion for $\Pi(\vec x)$ to be bounded away from zero at nonzero points of $\La$. But if it attains values however small, then we can talk about a corresponding Diophantine exponent. As before, we use $|\,\cdot\,|$ to denote the sup-norm.

\begin{definition} \label{d:lattice_exponent}
  We define the \emph{Diophantine exponent} of $\La$ as the supremum of real $\gamma$ such that the inequality
  \[ \Pi(\vec x)\leq |\vec x|^{-\gamma} \]
  admits infinitely many solutions in $\vec x\in\La$. We denote it by $\omega(\La)$.
\end{definition}

It follows immediately from Minkowski's convex body theorem that for each $\La$ we have the trivial inequality
\[ \omega(\La)\geq0. \]
At the same time we have $\omega(\La)=0$ whenever $N(\La)>0$. For instance, this holds for any lattice of a complete module in a totally real algebraic extension of $\Q$, which, by the way, makes the ``if'' part of the Oppenheim conjecture obvious.


There is another family of lattices for which we have $\omega(\La)=0$. It is provided by the famous subspace theorem proved by W.\,M.\,Schmidt \cite{schmidt_subspace} in 1972 (see also \cite{bombieri_gubler}).

\begin{theorem}[Schmidt'a subspace theorem, 1972]
  If $\pmb\ell_1(\vec z),\ldots,\pmb\ell_d(\vec z)$ are linearly independent linear forms in $d$ variables with algebraic coefficients, then for each $\e>0$ there are finitely many proper subspaces of $\Q^d$ containing all the integer points satisfying
  \vskip -3mm
  \[ \prod_{1\leq i\leq d}\big|\pmb\ell_i(\vec z)\big|<|\vec z|^{-\e}. \]
\end{theorem}

\begin{corollary} \label{cor:schmidt_finitely_many_points}
  Let $\pmb\ell_1(\vec z),\ldots,\pmb\ell_d(\vec z)$ be linearly independent linear forms in $d$ variables with algebraic coefficients. Suppose that for each $k$-tuple $(i_1,\ldots,i_k)$, $1\leq i_1<\ldots<i_k\leq d$, $1\leq k\leq d$,
  the coefficients of the multivector
  \[ \pmb\ell_{i_1}\wedge\ldots\wedge\pmb\ell_{i_k} \]
  are linearly independent over $\Q$.
  Then for each $\e>0$ there are only finitely many points $\vec z\in\Z^d$ satisfying
  \[ \prod_{1\leq i\leq d}\big|\pmb\ell_i(\vec z)\big|<|\vec z|^{-\e}. \]
\end{corollary}

\begin{proof}
  It follows from the restriction on the coefficients that for each $k$-dimensional rational subspace $\cL$ of $\R^d$ any $k$ of the given linear forms induce $k$ linearly independent linear forms in $\cL$.
  
  Let now $\cL$ be one of the rational subspaces mentioned in the subspace theorem. We may assume that $\dim\cL=d-1$ and identify it with $\R^{d-1}$ in such a way that $\cL\cap\Z^d$ turns into $\Z^{d-1}$. Then the initial forms $\pmb\ell_1(\vec z),\ldots,\pmb\ell_d(\vec z)$ induce new forms $\tilde{\pmb\ell}_1(\tilde{\vec z}),\ldots,\tilde{\pmb\ell}_d(\tilde{\vec z})$ in $d-1$ variables with algebraic coefficients, such that any $d-1$ of those forms are linearly independent. There is a constant $R$ depending only on the coefficients of the forms such that the set
  \[ \Big\{ \tilde{\vec z}\in\R^{d-1} \,\Big| \prod_{1\leq i\leq d}|\tilde{\pmb\ell}_i(\tilde{\vec z})|<|\tilde{\vec z}|^{-\e},\ |\tilde{\vec z}|>R \Big\} \]
  is contained in the union
  \[ \bigcup_{1\leq j\leq d}\Big\{ \tilde{\vec z}\in\R^{d-1} \,\Big| \prod_{\substack{1\leq i\leq d \\ i\neq j}}|\tilde{\pmb\ell}_i(\tilde{\vec z})|<|\tilde{\vec z}|^{-\e} \Big\}. \]
  The rest follows by induction, for the base case $d=2$ is obvious.
\end{proof}

\begin{corollary} \label{cor:schmidt_zero_omega}
  Let $\pmb\ell_1(\vec z),\ldots,\pmb\ell_d(\vec z)$ be as in Corollary \ref{cor:schmidt_finitely_many_points} and let $\La$ be be defined by \eqref{eq:lattice}. Then 
  \[ \omega(\La)=0. \]
\end{corollary}

\begin{proof}
  It suffices to notice that
  \begin{equation} \label{eq:equivalence_of_maxima}
    |\vec z|\asymp\max_{1\leq i\leq d}\big|\pmb\ell_i(\vec z)\big|
  \end{equation}
  and apply Corollary \ref{cor:schmidt_finitely_many_points}.
\end{proof}

It is reasonable to ask whether each positive value of $\omega(\La)$ can be attained, but the corresponding examples are yet to be constructed. As for now, we would like to pay attention to the transference phenomenon.


\section{Transference theorem}

Let $\La$ be an arbitrary lattice in $\R^d$. Consider the \emph{dual} lattice
\[ \La^\ast=\Big\{ \vec y\in\R^d \,\Big|\, \langle\vec y,\vec x\rangle\in\Z\text{ for each }\vec x\in\La \Big\}, \]
where $\langle\,\cdot\,,\,\cdot\,\rangle$ denotes the inner product. It appears that, same as in many other problems of Diophantine approximation, transference theorems can be proved, i.e. statements connecting $\omega(\La)$ and $\omega(\La^\ast)$.

Of course, if $d=2$ then $\La^\ast$ coincides up to a homothety with $\La$ rotated by $\pi/2$, so in the two-dimensional case we obviously have $\omega(\La)=\omega(\La^\ast)$.

\begin{theorem} \label{t:lattice_transference}
  Suppose $d\geq3$. Then
  \begin{equation} \label{eq:lattice_transference}
    \omega(\La)\geq\frac{\omega(\La^\ast)}{(d-1)^2+d(d-2)\omega(\La^\ast)}\,.
  \end{equation}
  Here we mean that if $\omega(\La^\ast)=\infty$, then $\omega(\La)\geq\dfrac{1}{d(d-2)}$\,.
\end{theorem}

We shall prove Theorem \ref{t:lattice_transference} with the help of the concept of a \emph{pseudo-compound parallelepiped} (see also \cite{schmidt_DA}) and a general transference theorem proved in \cite{german_evdokimov}. We give the definition in the simplest case, as this is the only case we need.

\begin{definition}
  Given positive numbers $\eta_1,\ldots,\eta_d$, consider the parallelepiped
  \[ \cP=\Big\{ \vec x=(x_1,\ldots,x_d)\in\R^d \,\Big|\, |x_i|\leq\eta_i,\ i=1,\ldots,d \Big\}. \]
  Then the parallelepiped
  \[ \cP^\ast=\Big\{ \vec x=(x_1,\ldots,x_d)\in\R^d \,\Big|\, |x_i|\leq\frac1{\eta_i}\prod_{1\leq j\leq d}\eta_j,\ i=1,\ldots,d \Big\} \]
  is called \emph{pseudo-compound} for $\Pi$.
\end{definition}

The transference principle discovered by Khintchine \cite{khintchine_palermo} for a particular case led eventually to the following rather general observation.

\begin{theorem}[G., Evdokimov, 2015] \label{t:german_evdokimov}
  Set $c=d^{\frac{1}{2(d-2)}}$ and suppose $\det\La=1$. Then
  \[ \cP^\ast\cap\La^\ast\neq\{\vec 0\}\implies c\cP\cap\La\neq\{\vec 0\}. \] 
\end{theorem}

Let us deduce Theorem \ref{t:lattice_transference} from Theorem \ref{t:german_evdokimov}.

\begin{proof}[Proof of Theorem \ref{t:lattice_transference}]
  Since $\omega(\La)$ is invariant under homotheties, we may suppose that $\det\La=1$.
  We consider two cases.
  
  \textbf{Case I:} There are no nonzero points of $\La^\ast$ in the coordinate planes. 
  
  Let us fix an arbitrary positive $\e$. Then there are infinitely many nonzero points $\vec u=(u_1,\ldots,u_d)\in\La^\ast$ such that
  \[ \Pi(\vec u)=|\vec u|^{-\gamma},\qquad \gamma=\gamma(\vec u)\geq
     \begin{cases}
       \omega(\La^\ast)-\e,\text{ if }\omega(\La^\ast)<\infty, \\
       1/\e,\text{ if }\omega(\La^\ast)=\infty.
     \end{cases} \]
  Let us consider any of those points and set
  \[ \cP_{\vec u}=\Big\{ \vec x=(x_1,\ldots,x_d)\in\R^d \,\Big|\, |x_i|\leq u_i,\ i=1,\ldots,d \Big\}. \]
  Since all the $u_i$ are nonzero, $\cP_{\vec u}$ is a non-degenerate parallelepiped. Moreover, $\cP_{\vec u}=\cP^\ast$ for
  \[ \cP=\Big\{ \vec x=(x_1,\ldots,x_d)\in\R^d \,\Big|\, |x_i|\leq\eta_i,\ i=1,\ldots,d \Big\}, \]
  where
  \[ \eta_i=u_i^{-1}\prod_{1\leq j\leq d}u_j^{\frac{1}{d-1}}. \]
  Hence by Theorem \ref{t:german_evdokimov} the parallelepiped $c\cP$ contains a nonzero point $\vec v=(v_1,\ldots,v_d)$ of $\La$. For this point we have
  \[ |\vec v|\leq c\max_{1\leq i\leq d}|\eta_i|\leq c\cdot\frac{\displaystyle\prod_{1\leq i\leq d}|\eta_i|}{\displaystyle\min_{1\leq i\leq d}|\eta_i|^{d-1}}=
     c\cdot\frac{\displaystyle\max_{1\leq i\leq d}|u_i|^{d-1}}{\displaystyle\prod_{1\leq i\leq d}|u_i|^{\frac{d-2}{d-1}}}=
     c\cdot\frac{|\vec u|^{d-1}}{\Pi(\vec u)^{\frac{d(d-2)}{d-1}}}=c|\vec u|^{d-1+\frac{d(d-2)}{d-1}\gamma} \]
  and
  \[ \Pi(\vec v)\leq c\cdot\prod_{1\leq i\leq d}|\eta_i|^{1/d}=c\cdot\prod_{1\leq i\leq d}|u_i|^{\frac{1}{d(d-1)}}=c\Pi(\vec u)^{\frac{1}{d-1}}=c|\vec u|^{-\frac{\gamma}{d-1}}. \]
  Thus,
  \begin{equation} \label{eq:Pi_v_vs_v}
    \Pi(\vec v)\leq c_1|\vec v|^{-\frac{\gamma}{(d-1)^2+d(d-2)\gamma}},\qquad c_1=c_1(d,\gamma).
  \end{equation}
  Notice that $|\vec u|$ may be however large. Hence $\min_{1\leq i\leq d}|\eta_i|$ may be however small. So, if there are no nonzero points of $\La$ in the coordinate planes, we get infinitely many points of $\La$ satisfying \eqref{eq:Pi_v_vs_v}, whence \eqref{eq:lattice_transference} follows. But if there is a nonzero point of $\La$ in a coordinate plane, then clearly $\omega(\La)=\infty$ and \eqref{eq:lattice_transference} holds trivially.

  \textbf{Case II:} There is a nonzero point of $\La^\ast$ in a coordinate plane.
  
  In this case we have $\omega(\La^\ast)=\infty$ and we are to show that 
  \begin{equation} \label{eq:degenerate_inequality}
    \omega(\La)\geq\frac{1}{d(d-2)}\,. 
  \end{equation}
  We may assume that there is a nonzero point $\vec u=(u_1,\ldots,u_d)\in\La^\ast$ with $u_d=0$. Then the $(d-1)$-dimensional subspace orthogonal to $\vec u$ contains the last coordinate axis and a sublattice $\Gamma\subset\La$ of rank $d-1$. Therefore, by Minkowski's convex body theorem there are infinitely many points $\vec v=(v_1,\ldots,v_d)\in\Gamma$ with $v_d\to\infty$ such that
  \[ \max_{1\leq i\leq d-1}|v_i|\leq c_2|v_d|^{-\frac{1}{d-2}},\qquad c_2=c_2(\vec u). \]
  For such $\vec v$ we have
  \[ \Pi(\vec v)\leq\Big(c_2^{d-1}|v_d|^{1-\frac{d-1}{d-2}}\Big)^{\frac{1}{d}}=c_2^{\frac{d-1}{d}}|\vec v|^{-\frac{1}{d(d-2)}}, \]
  whence \eqref{eq:degenerate_inequality} follows immediately.
\end{proof}

\section{Towards spectrum}

Same as in many other Diophantine problems (see \cite{laurent_2dim}, \cite{marnat_sharpness}, \cite{marnat_twisted}) it is reasonable to ask what subset of $(\R\cup\{\infty\})^2$ is formed by the pairs $(\omega(\La),\omega(\La^\ast))$ if $\La$ runs through the space of lattices in $\R^d$. 

As we have already noticed, for $d=2$ we have $\omega(\La)=\omega(\La^\ast)$. Besides that, in this simplest case everything can be described in terms of continued fractions (see \cite{cassels_DA}, \cite{khintchine_CF}), so, it is easy to see that for $d=2$ all the nonnegative values of $\omega(\La)$ are attained.

For $d\geq3$ we have the restrictions
\[ \omega(\La)\geq0,\quad\omega(\La^\ast)\geq0,\quad\omega(\La)\geq\frac{\omega(\La^\ast)}{(d-1)^2+d(d-2)\omega(\La^\ast)} \]
and it is interesting whether they determine the whole spectrum of $(\omega(\La),\omega(\La^\ast))$.

So far we know very little. We know examples of $\La$ with $\omega(\La)=\omega(\La^\ast)=0$. Those are either lattices with positive norm minimum \eqref{eq:norm_minimum}, or the ones provided by Corollaries \ref{cor:schmidt_finitely_many_points} and \ref{cor:schmidt_zero_omega}. Indeed, on one hand, it is well known (see \cite{german_norm_minima_I}, \cite{german_norm_minima_bordeaux}, \cite{skubenko_ndim}) that
\[ N(\La)>0\iff N(\La^\ast)>0. \]
On the other hand, if $\pmb\ell_1(\vec z),\ldots,\pmb\ell_d(\vec z)$ are linearly independent and $\pmb\ell_1^\ast(\vec z),\ldots,\pmb\ell_d^\ast(\vec z)$ are the dual linear forms, then the coefficients of $\pmb\ell_{i_1}(\vec z)\wedge\ldots\wedge\pmb\ell_{i_k}(\vec z)$ coincide up to signs with those of $\pmb\ell_{i_{k+1}}^\ast(\vec z)\wedge\ldots\wedge\pmb\ell_{i_d}^\ast(\vec z)$, where $(i_1,\ldots,i_d)$ is a permutation of $(1,\ldots,d)$. So, those sets of coefficients are simultaneously linearly independent over $\Q$, which means that if $\La$ satisfies the hypothesis of Corollary \ref{cor:schmidt_finitely_many_points}, then so does $\La^\ast$.

It appears that the subspace theorem also provides examples of $\La$ such that
\[ \omega(\La)=\frac{1}{d(d-2)}\,,\quad\omega(\La^\ast)=\infty, \]
proving thus sharpness of Theorem \ref{t:lattice_transference} in one boundary case. 

\begin{theorem} \label{t:sharpness_from_schmidt}
  Let $\pmb\ell_1(\vec z),\ldots,\pmb\ell_d(\vec z)$ be linearly independent linear forms in $d$ variables with algebraic coefficients. Suppose that the first coefficient of the multivector 
  \[ \pmb\ell_1\wedge\ldots\wedge\pmb\ell_{d-1} \] 
  equals zero and that the rest of them are linearly independent over $\Q$. Suppose also that for each $k$-tuple $(i_1,\ldots,i_k)$ different from $(1,\ldots,d-1)$, $1\leq i_1<\ldots<i_k\leq d$, $1\leq k\leq d$, the coefficients of
  \[ \pmb\ell_{i_1}\wedge\ldots\wedge\pmb\ell_{i_k} \]
  are linearly independent over $\Q$. Let $\La$ be defined by \eqref{eq:lattice}. Then 
  \begin{equation}
    \omega(\La)=\frac{1}{d(d-2)}\,,\quad\omega(\La^\ast)=\infty.
  \end{equation}
\end{theorem}

\begin{proof}
  Since the first coefficient of $\pmb\ell_1\wedge\ldots\wedge\pmb\ell_{d-1}$ is zero, there is a nonzero point of $\La^\ast$ in
  \[ \cL_1=\Big\{ \vec z=(z_1,\ldots,z_d)\in\R^d \,\Big|\, z_d=0 \Big\}. \]
  Hence $\omega(\La^\ast)=\infty$, so, by Theorem \ref{t:lattice_transference} it suffices to show that
  \begin{equation} \label{eq:inverse_degenerate_inequality}
    \omega(\La)\leq\frac{1}{d(d-2)}\,.
  \end{equation}
  It follows from the hypothesis that if $\cL$ is an arbitrary rational subspace of $\R^d$ different from $\cL_1$, $\dim\cL=k$, then any $k$ of the given linear forms induce $k$ linearly independent linear forms in $\cL$. Thus, repeating the argument of Corollary \ref{cor:schmidt_finitely_many_points} one can show that for any $\e>0$ and all $\vec z\in\Z^d\backslash\cL_1$ we have
  \begin{equation} \label{eq:outside_L1}
    \prod_{1\leq i\leq d}\big|\pmb\ell_i(\vec z)\big|\geq c_3|\vec z|^{-\e},\qquad c_3=c_3(\e,\pmb\ell_1,\ldots,\pmb\ell_d).
  \end{equation}
  
  As for $\cL_1$, by the hypothesis any $d-1$ forms 
  \[ \pmb\ell_{i_1}(\vec z),\ldots,\pmb\ell_{i_{d-2}}(\vec z),\pmb\ell_d(\vec z),\qquad 1\leq i_1<\ldots<i_{d-2}\leq d-1, \] 
  induce linearly independent forms $\tilde{\pmb\ell}_1(\tilde{\vec z}),\ldots,\tilde{\pmb\ell}_{d-1}(\tilde{\vec z})$ in $\cL_1$ which satisfy the hypothesis of Corollary \ref{cor:schmidt_finitely_many_points}. Therefore, for each $\e>0$, each $j\in\{1,\ldots,d-1\}$ and each nonzero $\vec z\in\cL_1$ we have
  \[ \prod_{\substack{1\leq i\leq d \\ i\neq j}}\big|\pmb\ell_i(\vec z)\big|\geq c_4|\vec z|^{-\e},\qquad c_4=c_4(\e,\pmb\ell_1,\ldots,\pmb\ell_d). \]
  Hence
  \[ \big|\pmb\ell_d(\vec z)\big|\Bigg(\prod_{1\leq i\leq d}\big|\pmb\ell_i(\vec z)\big|\Bigg)^{d-2}=
     \prod_{1\leq j\leq d-1}\Bigg(\prod_{\substack{1\leq i\leq d \\ i\neq j}}\big|\pmb\ell_i(\vec z)\big|\Bigg)\geq c_4^{d-1}|\vec z|^{-(d-1)\e}. \]
  Thus, taking into account \eqref{eq:equivalence_of_maxima}, we see that for each $\e>0$ and each nonzero $\vec z\in\cL_1$ we have
  \begin{equation} \label{eq:inside_L1}
    \prod_{1\leq i\leq d}\big|\pmb\ell_i(\vec z)\big|\geq c_5|\vec z|^{-\frac{1}{d-2}-\e},\qquad c_5=c_5(\e,\pmb\ell_1,\ldots,\pmb\ell_d).
  \end{equation}
  Once again taking into account \eqref{eq:equivalence_of_maxima}, we get from \eqref{eq:outside_L1} and \eqref{eq:inside_L1} that for each $\e>0$ and each $\vec x\in\La$
  \[ \Pi(\vec x)\geq c_6|\vec x|^{-\frac{1}{d(d-2)}-\e},\qquad c_6=c_6(\e,\La), \]
  whence \eqref{eq:inverse_degenerate_inequality} follows.
\end{proof}

  It is not difficult to see that similar argument can be used to construct lattices with $\omega(\La^\ast)=\infty$ and $\omega(\La)$ equal to any of the values
  \begin{equation} \label{eq:spectrum_from_schmidt}
    \frac{k(d-k-l)}{dl}\,,\qquad
      \begin{array}{l}
        k\in\{1,\ldots,d-2\}, \\
        \hskip 1mm
        l\in\{1,\ldots,d-k-1\}.
      \end{array}
  \end{equation}
  To do so one should construct $\pmb\ell_1(\vec z),\ldots,\pmb\ell_d(\vec z)$ with algebraic coefficients such that the $k$-dimensional subspace determined by
  \[ \pmb\ell_1(\vec z)=\ldots=\pmb\ell_{d-k}(\vec z)=0 \]
  is contained in a rational subspace of dimension $k+l\leq d-1$, but is not contained in any rational subspace of smaller dimension. Then Minkowski's convex body theorem can be applied to prove the inequality
  \[ \omega(\La)\geq\frac{k(d-k-l)}{dl}\,, \]
  and the subspace theorem to prove the inverse one.
  
  However, \eqref{eq:spectrum_from_schmidt} are the only nontrivial values of $\omega(\La)$ this method gives. Even the question whether there are lattices with finite nonzero $\omega(\La)$ different from \eqref{eq:spectrum_from_schmidt} is still open.
  

%
%


\begin{thebibliography}{99}

\bibitem
    {cassels_GN} 
    \textsc{J.\,W.\,S.\,Cassels} 
    \textit{An introduction to the geometry of numbers.} 
    Springer (1997).
\bibitem
    {gruber_lekkerkerker}
    \textsc{P.\,M.\,Gruber, C.\,G.\,Lekkerkerker} 
    \textit{Geometry of numbers.} 
    North-Holland Math. Library, \textbf{37}, Elsevier (1987).
\bibitem
    {cassels_swinnerton_dyer} 
    \textsc{J.\,W.\,S.\,Cassels, H.\,P.\,F.\,Swinnerton--Dyer}
    \textit{On the product of three homogeneous linear forms and indefinite ternary quadratic forms.}
    Phil. Trans. Royal Soc. London, {\bf A 248} (1955), 73--96.
\bibitem
    {borevich_shafarevich} 
    \textsc{Z.\,I.\,Borevich, I.\,R.\,Shafarevich}
    \textit{Number theory}.
    NY Academic Press (1966).
\bibitem
    {schmidt_subspace}
    \textsc{W.\,M.\,Schmidt}
    \textit{Norm form equations.}
    Ann. Math., {\bf 96} (1972), 526--551.
\bibitem
    {bombieri_gubler}
    \textsc{E.\,Bombieri, W.\,Gubler}
    \textit{Heights in Diophantine geometry.}
    Cambridge Univ. Press (2006).
\bibitem
    {schmidt_DA} 
    \textsc{W.\,M.\,Schmidt} 
    \textit{Diophantine Approximation.} 
    Lecture Notes in Math., \textbf{785}, Springer-Verlag (1980).
\bibitem
    {german_evdokimov}
    \textsc{O.\,N.\,German, K.\,G.\,Evdokimov}
    \textit{A strengthening of Mahler's transference theorem.}
    Izvestiya: Mathematics, \textbf{79}:1 (2015), 60--73.
\bibitem
    {khintchine_palermo}
    \textsc{A.\,Ya.\,Khintchine}
    \textit{\"Uber eine Klasse linearer Diophantischer Approximationen.}
    Rend. Sirc. Mat. Palermo, \textbf{50} (1926), 170--195.
\bibitem
    {laurent_2dim}
    \textsc{M.\,Laurent}
    \textit{Exponents of Diophantine approximation in dimension two.}
    Canad. J. Math., \textbf{61} (2009), 165--189.
\bibitem
    {marnat_sharpness}
    \textsc{A.\,Marnat}
    \textit{About Jarník's-type relation in higher dimension.}
    arXiv:1510.06334.
\bibitem
    {marnat_twisted}
    \textsc{A.\,Marnat}
    \textit{There is no analogue to Jarník’s relation for twisted Diophantine approximation.}
    Monats. Math, to appear; arXiv:1409.6665.
\bibitem
    {cassels_DA} 
    \textsc{J.\,W.\,S.\,Cassels} 
    \textit{An introduction to Diophantine approximation.} 
    Cambridge Univ. Press (1957).
\bibitem
    {khintchine_CF} 
    \textsc{A.\,Ya.\,Khintchine} 
    \textit{Continued fractions.} 
    Dover Publications (1997).
\bibitem
    {german_norm_minima_I}  
    \textsc{O.\,N.\,German}
    \textit{Sails and norm minima of lattices.}
    Mat. Sb. {\bf 196}:3 (2005), 31--60;
\bibitem
    {german_norm_minima_bordeaux} 
    \textsc{O.\,N.\,German}
    \textit{Klein polyhedra and lattices with positive norm minima.}
    JTNB {\bf 19} (2007), 177--192.
\bibitem
    {skubenko_ndim} 
    \textsc{B.\,F.\,Skubenko}
    \textit{Minima of decomposable forms of degree $n$ of $n$ variables for $n\geq 3$.}
    Zapiski nauch. sem. LOMI, {\bf 183}, 1990.

\end{thebibliography}
\end{document}